\documentclass[11pt]{article}

\usepackage[margin=1in]{geometry}
\usepackage{amsmath, amssymb, amsthm, mathtools}
\usepackage{hyperref}
\usepackage{microtype}
\usepackage{xcolor}
\usepackage{graphicx}

\newtheorem{theorem}{Theorem}

\newtheorem{proposition}[theorem]{Proposition}

\theoremstyle{definition}
\newtheorem{definition}[theorem]{Definition}

\newcommand{\E}{\mathbb{E}}
\newcommand{\Pbb}{\mathbb{P}}
\newcommand{\binomtwo}[1]{\binom{#1}{2}}
\newcommand{\binomthree}[1]{\binom{#1}{3}}

\title{An update on multicolor Ramsey lower bounds}

\author{
Marcelo Campos%
\thanks{Department of Mathematics, IMPA, Rio de Janeiro, Brazil. Email: \texttt{marcelo.campos@impa.br}.}
\and
Cosmin Pohoata%
\thanks{Department of Mathematics, Emory University, Atlanta, GA. Email: \texttt{cosmin.pohoata@emory.edu}.}
}

\date{}

\begin{document}
\maketitle
\begin{abstract}
Building upon previous works by Conlon-Ferber and Wigderson, Sawin showed a few years ago that upper bounds on the minimum density of independent sets in a
$K_t$-free $G$ can be used to provide lower bounds for multicolor Ramsey numbers. 

In this note, we observe how a further improved upper bound on this parameter directly follows from a recent spherical random geometric graph construction of Ma-Shen-Xie. As a consequence, we derive a small exponential improvement over the best known lower bounds for multicolor Ramsey numbers.
\end{abstract}

\section{Introduction}

For integers $\ell,t\ge 2$, let $r(t;\ell)$ denote the least $N$ such that every $\ell$-coloring of the edges of
$K_N$ contains a monochromatic copy of $K_t$. The most well-studied case is that of $\ell=2$, where the bounds
$$2^{t/2} \leq r(t;2) \leq 2^{2t}$$
were originally proved by Erd\H{o}s \cite{erdos1947remarks} and Erd\H{o}s-Szekeres \cite{erdos1935combinatorial} in 1947 and 1935, respectively. The lower bound is a particularly celebrated probabilistic construction, which in fact initiated the development of the probabilistic method in combinatorics and related areas. Despite decades of effort, this lower bound has only been improved so far by a multiplicative constant (see \cite{spencer1975ramsey}). For the upper bound, the first exponential improvement was established only a couple of years ago by the first author, together with Griffiths, Morris and Sahasrabudhe \cite{campos2026exponential}. 

For $\ell\ge 3$, a straightforward modification of the Erd\H{o}s--Szekeres ``neighbourhood chasing'' argument yields the general upper bound $r(t;\ell)\le \ell^{\ell t}$. Using a new geometric approach and ideas related to the work from \cite{campos2026exponential}, a recent paper \cite{balister2026upper} provided an exponential improvement over this bound for all $\ell \geq 2$. In the other direction, the usual random colouring argument gives $r(t;\ell)\ge \ell^{t/2}$.
This was originally improved by Abott~\cite{Abbott}, who used an iterated product coloring to show that $r(t;\ell)\ge 2^{\ell t/4}$. This was more recently improved in a remarkable short paper by Conlon and Ferber~\cite{CF}, who introduced a randomized algebraic construction to improve on this bound for $\ell=3$, and then leveraged Abott's product coloring trick to get an improvement for all $\ell \geq 3$.  Their lower bound was subsequently improved for $\ell \geq 4$ by Wigderson \cite{Wig}, who replaced the Abott's trick with a random homomorphism argument. Finally, Sawin \cite{Sawin} improved on the bounds of Conlon-Ferber and Wigderson for all $\ell \geq 3$, by replacing the algebraic part of the Conlon-Ferber construction with an appropriate copy of the Erd\H{o}s-Renyi random graph. 

In this note, we show that one can further improve the lower bound of Sawin for all $\ell \geq 3$, by replacing the copy of the Erd\H{o}s-Renyi random graph with the random geometric graph, used recently by Ma, Shen and Xie in \cite{MSX}. 

Before we say more about this graph and where it comes in, let us first describe this story in a bit more detail. Following Sawin \cite{Sawin}, given integers $s,t\ge 1$, let us start by defining $c_{s,t}$ to be the infimum over all graphs $G$ with no $K_t$ of the probability that
$s$ i.i.d.\ uniform random vertices $v_1,\dots,v_s\in V(G)$ form an independent set (allowing repetitions). It turns out that upper bounds for this parameter translate very naturally into lower bounds for multicolor Ramsey numbers. 

\begin{theorem}[Sawin]
\label{prop:sawin-reduction}
For all $\ell,t\ge 2$,
\[
r(t;\ell)\ \ge\ c_{t,t}^{-(\ell-2)/t}\,2^{(t-1)/2}.
\]
\end{theorem}
Indeed, fix a $K_t$-free graph $G$ for which a random $t$-tuple of vertices forms an independent set with probability
$\approx c_{t,t}$. The idea of Wigderson \cite{Wig}, which goes back in some sense to the work of Alon and R\"odl \cite{AlonRodl}, is that one can build a coloring of $K_N$ by ``trying'' $\ell-2$ random homomorphisms
$f_1,\dots,f_{\ell-2}:V(K_N)\to V(G)$ in order: we give an edge $xy$ the first color $i$ for which
$f_i(x)f_i(y)$ lands on an edge of $G$, and if none of the $\ell-2$ tests succeeds we assign $xy$ one of the
two remaining colors uniformly at random.
By construction, the first $\ell-2$ colors cannot contain a $K_t$ (since $G$ itself is $K_t$-free).
So the only possible monochromatic $K_t$ would have to use one of the last two colors. For a fixed
$t$-set $S$, this requires two independent-looking conditions: for every $i\le \ell-2$, the image $f_i(S)$
must be an independent set in $G$ (costing about $c_{t,t}^{\ell-2}$), and then the final random two-coloring
must make all $\binom{t}{2}$ edges of $S$ agree (costing about $2^{1-\binom{t}{2}}$).
A union bound over the $\binom{N}{t}$ choices of $S$ shows that if
\[
\binom{N}{t}\,c_{t,t}^{\ell-2}\,2^{1-\binom{t}{2}}<1,
\]
then with positive probability the resulting $\ell$-coloring contains no monochromatic $K_t$.
It is not difficult to see that $N\asymp c_{t,t}^{-(\ell-2)/t}2^{(t-1)/2}$ satisfies this inequality, therefore it follows that $r(t;\ell)\ \ge\ c_{t,t}^{-(\ell-2)/t}\,2^{(t-1)/2}$. 

\medskip
\noindent {\bf{Upper bounds for $c_{t,t}$.}} In turn, obtaining upper bounds on $c_{t,t}$ amounts to exhibiting graphs $G$ with no $K_t$ for which a random $t$-tuple of vertices is
unlikely to form an independent set.
Indeed, by definition, $c_{t,t}$ is the infimum of this probability over all $K_t$-free graphs,
so any such example immediately yields an upper bound.
The problem therefore becomes one of constructing graphs that simultaneously avoid large cliques
and yet contain very few large independent sets, at least in the probabilistic sense captured by
$c_{t,t}$. In \cite{Sawin}, Sawin approached this problem by taking $G$ to be an Erd\H{o}s--R\'enyi random graph on $M$ vertices,
with edge density slightly below $1/2$, and then conditioning on the event that $G$ contains no $K_t$. For every $t\ge 2$ and every $p\in(0,1)$, this led to the following estimate:
\begin{equation}
\label{eq:sawin-ctt-general}
c_{t,t}\ \le\ 
\exp\left(\,\frac{\,(4\log(1-p)-\log p)\,\log p}{8\log(1-p)}t^2\;+\;o(t^2)\right).
\end{equation}
In particular, optimizing over $p$ the optimum is attained numerically at $p\approx 0.454997$. Combinining this with Theorem \ref{prop:sawin-reduction}, Sawin then obtained the multicolor Ramsey lower bound
\begin{equation}
\label{eq:sawin-ramsey}
r(t;\ell)\ \ge\ 2^{\,0.383796  (\ell-2)t\;+\;\frac{t}{2}\;+\;o(t)},
\end{equation}
for every fixed $\ell \geq 3$. 

Our main observation is that one can get an exponential improvement over \eqref{eq:sawin-ramsey} by instead taking $G$ to be the complement of a
\emph{dense random geometric graph} in high dimension. This graph was recently used by Ma, Shen and Xie~\cite{MSX} in a new exciting lower bound for (two-color) off-diagonal Ramsey numbers. Even more recently, a variant of the model which is much simpler to analyze was introduced by Hunter, Milojevi\'c and Sudakov~\cite{HMS} in order to give a shorter proof of the main result of~\cite{MSX}. The key feature in these models is that edges are no longer independent: geometric constraints induce
negative correlations among non-edges and positive correlations among edges.
As an important consequence, large independent sets are suppressed compared to the Erd\H{o}s--R\'enyi
model, while large cliques become more likely. By choosing the dimension appropriately and slightly adjusting the edge density, the suppression of
independent sets outweighs the increased likelihood of cliques, leading to a modest but genuine
improvement over the Erd\H{o}s--R\'enyi construction in the bounds for $c_{t,t}$, and hence for multicolor Ramsey numbers $r(t;\ell)$ for all $l \geq 3$. 

Our construction balances two geometric corrections in the bounds
of Ma, Shen and Xie~\cite{MSX}: an \emph{extra suppression} of cliques, and an \emph{extra enhancement} of independent sets on a fixed vertex
subset (which works against us). From this we may obtain the following theorem.

\begin{theorem} \label{cor:main-numerics}
For all $p\in(0.42,0.5)$ there exists $\gamma>0$ such that,
\[
c_{t,t}\leq \exp\left(\,\left(\frac{\,(4\log(1-p)-\log p)\,\log p}{8\log(1-p)}-\gamma\right)t^2\;+\;o(t^2)\right).
\]
\end{theorem}
One surprising feature of the construction is that it doesn't improve over $G(n,p)$ for the full range of $p\in (0,1/2)$, but only when $p\in (0.42,0.5)$. This turns out to be enough since Sawin's optimal value of $p$ lies in this interval.


\medskip
\noindent {\bf{Acknowledgments.}} MC was supported by Serrapilheira (grant R-2412-51283). CP was supported by NSF grant DMS-2246659. The authors would like to thank Rob Morris for useful comments on an earlier version of the manuscript.

\section{Geometric random graph model}
\label{sec:models}

Ma--Shen--Xie~\cite{MSX} used a dense random geometric graph model on the high-dimensional sphere.
One convenient way to parameterize this model is to fix an edge density $p$ and then choose an
appropriate distance/angle threshold.

\begin{definition}[Spherical random geometric graph $H(M,d,p)$]
Fix $M,d\ge 1$ and $p\in(0,1/2)$. Let $u_1,\dots,u_M$ be i.i.d.\ uniform points on the unit sphere
$S^{d-1}\subset\mathbb{R}^d$. Let $\tau=\tau(d,p)$ be the unique real number such that for independent
$u,u'\sim \mathrm{Unif}(S^{d-1})$ we have $\Pbb(\langle u,u'\rangle<\tau)=p$.
Define $H(M,d,p)$ on vertex set $[n]$ by putting an edge $ij$ iff
$\langle u_i,u_j\rangle< \tau$.
\end{definition}

Ma--Shen--Xie~\cite{MSX} use this (equivalently, a distance-threshold formulation) as the underlying random graph
in a two-color Ramsey construction. 
The point of the
geometric model is that, while the edge density is still about $1-p$, higher-order edge correlations differ from
the Erd\H{o}s--R\'enyi model in a way that can be exploited to obtain improved Ramsey lower bounds in the off-diagonal case. Very recently, Hunter--Milojevi\'c--Sudakov~\cite{HMS} introduced a Gaussian analogue of the above construction, replacing uniform spherical points with gaussian vectors. Their model exibits the same behavior, but the analysis becomes significantly simpler. 

We next record the main estimate that we will need. Like above, let $\varphi$ denote the standard normal density and set $a \coloneqq \varphi(c_p)=\frac{e^{-c_p^2/2}}{\sqrt{2\pi}}$. The following convenient proposition follows directly from Lemma 9.1 in~\cite{MSX} and is also proved in Corollary 3.1~\cite{HMS}.

\begin{proposition}
\label{prop:prob_bounds}
Fix $p\in(0,1/2)$. There exist constants $D_0=D_0(p)$ and $K=K(p)$
such that the following holds for all integers $t\ge 2$, all $D\ge D_0$, and $d=D^2t^2$.
For every $1\le r\le t$,
\begin{align}
\label{eq:Pred-unif}
\Pbb\big(H(r,d,p)\text{ is an independent set}\big)
&\le p^{\binomtwo{r}}
\exp\!\left(
-\frac{a^3}{p^3\sqrt d}\binomthree{r}
\;+\;
K\,\frac{r^3}{D\sqrt d}
\right),\\[1ex]
\label{eq:Pblue-unif}
\Pbb\big(H(r,d,p)\text{ is a clique}\big)
&\le (1-p)^{\binomtwo{r}}
\exp\!\left(
+\frac{a^3}{(1-p)^3\sqrt d}\binomthree{r}
\;+\;
K\,\frac{r^3}{D\sqrt d}
\right).
\end{align}
\end{proposition}

We will deduce the following main theorem from Proposition~\ref{prop:prob_bounds}. 

\begin{theorem}
\label{thm:geom-sawin}
Fix $\delta>0$ and $p\in(0,1/2)$ and constants $D\ge D_0(p)$ and $K=K(p)$. For each $t\ge 2$ set $d=D^2t^2$.
Let $c_p>0$ be defined by $\Pbb(Z\le -c_p)=p$ for $Z\sim N(0,1)$.

Let $\varphi$ denote the standard normal density, and set $a=\varphi(c_p)$.
Define
\[
\kappa(p)=\frac{a^3}{6p^3},
\qquad
\lambda(p)=\frac{a^3}{6(1-p)^3}.
\]
Finally let $M=\lfloor e^{mt}\rfloor$ where
\[
m=-\frac12\log p+\frac{\kappa(p)}{D}-\frac{K}{D^2}-\delta\, .
\]

Then letting $H:=H(M,d,p)$ we have that $\Pbb(H\text{ is }K_t\text{-free})=1-o(1)$ and, writing $v_1,\dots,v_t$ for i.i.d.\ uniform
vertices in $[M]$ (with replacement),
\[
\Pbb\big(\{v_1,\dots,v_t\}\text{ is independent in }H\big)
\le \exp(\,-\alpha(p,D)t^2+o(t^2)) ,
\]
where
\[
\alpha(p,D)
= \min_{\theta\in[0,1]}
\left(
(1-\theta)\Big(-\tfrac12\log p+\tfrac{\kappa(p)}{D}\Big)
-\frac{\theta^2}{2}\log(1-p)-\frac{\lambda(p)}{D}\theta^3
-(2-\theta) \frac{K}{D^2}\right)-\delta .
\]
Consequently,
\[
c_{t,t}\ \le\ \exp(\,-\alpha(p,D)t^2+o(t^2)).
\]
\end{theorem}
Theorem~\ref{thm:geom-sawin} constructs a graph $H$ on $M=e^{mt}$ vertices which is $K_t$-free with high
probability and for which $t$ i.i.d.\ uniform vertices span an independent set with probability at most
$\exp(-\alpha t^2+o(t^2))$. This theorem will quickly imply Theorem~\ref{cor:main-numerics} by approximating $\alpha(p,D)$ for $D$ large. 

The specific value of $\gamma$ in Theorem \ref{cor:main-numerics} depends on the values of $K$ and $D_0$ in our Theorem~\ref{thm:geom-sawin}, which in turn directly depend on Proposition~\ref{prop:prob_bounds}. By optimizing these values one would be immediately be able to obtain explicit bounds in Theorem~\ref{cor:main-numerics}. For instance the Hunter, Milojevi\'c, Sudakov analysis from \cite{HMS} might be very useful for this purpose. Combining Theorem~\ref{cor:main-numerics} with Theorem~\ref{prop:sawin-reduction} immediately yields new slightly better multicolor Ramsey lower bounds.

\section{Bounding $c_{t,t}$ using the random geometric graph}
\label{sec:proof}

We now prove Theorem~\ref{thm:geom-sawin}. The argument follows from estimating the expected number of cliques and of independent sets in the geometric random graph, and noticing that they will be smaller than in the Erd\H{o}s-Rényi random graph. 

Recall the setup: we fixed $p \in (0,1)$ and the constants
\[
\kappa(p)\coloneqq \frac{a^3}{6p^3},
\qquad
\lambda(p)\coloneqq \frac{a^3}{6(1-p)^3}.
\]
Let $C_0=2$ and $D\ge D_0(p,C_0)$, where $D_0$ is the constant from Proposition \ref{prop:prob_bounds}. Set $d=D^2t^2$ and let $H\sim H(M,d,p)$. We want to show that there exists a choice of $M$ of the form $\log M = m t$ with
\begin{equation}
\label{eq:m-def}
m\ =\ -\frac12\log p\ +\ \frac{\kappa(p)}{D}\ +\ O_p\!\left(\frac{1}{D^2}\right)
\end{equation}
such that $\Pbb(H\text{ contains a }K_t)=o(1)$, as well as
\begin{equation}
\label{eq:ctt-bound}
\Pbb\big(\{v_1,\dots,v_t\}\text{ is independent in }H\big)
\le \max_{1\le k\le t}\left(
M^{-(t-k)}\,(1-p)^{\binomtwo{k}}\,
\exp\left(\frac{6\lambda(p)}{D}\cdot \frac{\binomthree{k}}{t} \;+\; K \frac{t^2}{D^2}\right)
\right),
\end{equation}
where $v_1,\dots,v_t$ are i.i.d.\ uniform vertices of $H$ (with replacement). Consequently, there exists $\alpha_{\mathrm{true}}(p,D)$ such that
\[
c_{t,t}\ \le\ \exp(\,-\;\alpha_{\mathrm{true}}(p,D)\,t^2\;+\;o(t^2)).
\]
Moreover, if we define the ``formal'' exponent
\begin{equation}
\label{eq:alpha-formal}
\alpha_{\mathrm{formal}}(p,D)\coloneqq \min_{\theta\in[0,1]}
\left(
(1-\theta)\Big(-\tfrac12\log p+\tfrac{\kappa(p)}{D}\Big)
-\frac{\theta^2}{2}\log(1-p)-\frac{\lambda(p)}{D}\theta^3
-(2-\theta)\frac{K}{D^2}\right),
\end{equation}
then
\begin{equation}
\label{eq:alpha-error}
\alpha_{\mathrm{true}}(p,D)\ =\ \alpha_{\mathrm{formal}}(p,D)\ +o(1).
\end{equation}

\begin{proof}
Let $\mathcal A$ be the event that $H$ is $K_t$-free. For a (deterministic) graph $H$ on vertex set $[M]$, define
\[
q(H)\ :=\ \Pbb(\{v_1,\dots,v_t\}\ \text{is an independent set in }H),
\]
where $v_1,\dots,v_t$ are i.i.d.\ uniform in $[M]$ (with replacement).
If $H$ is $K_t$-free then, by definition of $c_{t,t}$ as an infimum over $K_t$-free graphs,
we have $c_{t,t}\le q(H)$. Hence on $\mathcal A$ we have $c_{t,t}\le q(H)$, and therefore
\begin{equation}
\label{eq:conditioning-ineq}
c_{t,t}\ \le\ \E[q(H)\mid \mathcal A]
\ =\ \frac{\E[q(H)\mathbf 1_{\mathcal A}]}{\Pbb(\mathcal A)}
\ \le\ \frac{\E[q(H)]}{\Pbb(\mathcal A)}.
\end{equation}
Thus it suffices to choose $M$ so that $\Pbb(\mathcal A)=1-o(1)$ and then to bound $\E[q(H)]$.

\medskip
\textbf{Showing $H$ is $K_t$-free.} Let $X$ be the number of $K_t$'s in $H$. Then
\[
\Pbb(H\text{ contains a }K_t)=\Pbb(X\ge 1)\le \E[X].
\]
By linearity of expectation and symmetry,
\begin{equation}
\label{eq:EX-clique}
\E[X]\ =\ \binom{M}{t}\,\Pbb\big(H(t,d,p)\text{ is a clique}\big).
\end{equation}
Applying Proposition~\ref{prop:prob_bounds} with $r=t$ gives
\[
\Pbb\big(H(t,d,p)\text{ is a clique}\big)
\le p^{\binomtwo{t}}
\exp\!\left(
-\frac{a^3}{p^3\sqrt d}\binomthree{t}
+K\,\frac{t^3}{D\sqrt d}
\right).
\]
Since $\sqrt d=Dt$ and $\binomthree{t}=\frac{t^3}{6}+O(t^2)$, we have
\[
\frac{a^3}{p^3\sqrt d}\binomthree{t}
=\frac{a^3}{p^3}\cdot \frac{t^3/6+O(t^2)}{Dt}
=\frac{\kappa(p)}{D}\,t^2+o(t^2),
\]
and also $K\,\frac{t^3}{D\sqrt d}=K\,\frac{t^2}{D^2}$.
This yields
\begin{equation}
\label{eq:clique-prob-t}
\Pbb\big(H(t,d,p)\text{ is a clique}\big)
\le p^{\binomtwo{t}}\,
\exp\left(-\frac{\kappa(p)}{D}t^2\;+\;K\,\frac{t^2}{D^2}\;+\; o(t^2)\right).
\end{equation}

For any $\delta>0$ we may set
\begin{equation}
\label{eq:m-choice}
m\ :=\ -\frac12\log p+\frac{\kappa(p)}{D}-\frac{K}{D^2}-\delta,
\qquad
M:=\left\lfloor \exp(mt)\right\rfloor.
\end{equation}
Using that $M^t=\exp(mt^2+O(t))$ and $p^{\binomtwo{t}}=\exp(\binomtwo{t}\log p)=\exp((\frac12t^2+O(t))\log p)$, we combine
\eqref{eq:EX-clique} and \eqref{eq:clique-prob-t} to get
\[
\E[X]
\le 
\exp(mt^2)\cdot \exp\left(\frac12t^2\log p\right)\cdot
\exp\left(-\frac{\kappa(p)}{D}t^2+K\frac{t^2}{D^2}\right).
\]
Using the definition of $m$ in \eqref{eq:m-choice}, this becomes
\[
\E[X]\ \le \exp\left(-\delta t^2 +o(t^2)\right)\ =\ o(1).
\]
Hence $\Pbb(\mathcal A)=1-o(1)$.

\medskip
\textbf{Bounding $q(H)$.} Next, note that
\begin{align*}
   \E_H[q(H)]
&=\E_H\Pbb(\{v_1,\dots,v_t\}\text{ independent in }H)\\
&\leq t! \; M^{-t}\sum_{k=1}^t M^{k}\cdot \,\Pbb_H(H(k,d,p)\text{ is independent}). 
\end{align*}

Now apply Proposition \ref{prop:prob_bounds} with $r=k\le t\le 2t$:
\[
\Pbb(H(k,d,p)\text{ is independent})
\le (1-p)^{\binomtwo{k}}
\exp\!\left(
\frac{a^3}{(1-p)^3\sqrt d}\binomthree{k}
+K\,\frac{k^3}{D\sqrt d}
\right).
\]
Since $\sqrt d=Dt$, the leading correction becomes
\[
\exp\!\left(\frac{a^3}{(1-p)^3\sqrt d}\binomthree{k}\right)
=
\exp\left(\frac{a^3}{(1-p)^3}\cdot \frac{1}{Dt}\binomthree{k}\right)
=
\exp\left(\frac{6\lambda(p)}{D}\cdot\frac{\binomthree{k}}{t}\right).
\]
Moreover, $K\frac{k^3}{D\sqrt d}=K\frac{k^3}{D\cdot Dt}\le K\frac{t^2}{D^2}$ which gives the desired bound.

\medskip\noindent\textbf{Consequences for $c_{t,t}$ and the exponent $\alpha$.}
Since $\Pbb(\mathcal A)=1-o(1)$, combining \eqref{eq:conditioning-ineq} therefore implies that
\[
c_{t,t}\ \le\ (1+o(1))\,\E[q(H)]
\ \le\ (1+o(1))\,t!\max_{1\le k\le t}\left(
M^{-(t-k)}(1-p)^{\binomtwo{k}}
\exp\left(\frac{6\lambda(p)}{D}\cdot\frac{\binomthree{k}}{t}+K \frac{t^2}{D^2}\right)
\right).
\]
Using that $t!=\exp(o(t^2))$ we deduce~\eqref{eq:ctt-bound}, therefore
\[
c_{t,t}\ \le\ \exp(-\alpha_{\mathrm{true}}(p,D)t^2+o(t^2))
\]
for some $\alpha_{\mathrm{true}}(p,D)$.

Finally, to identify $\alpha_{\mathrm{true}}(p,D)$ set $k=\lfloor \theta t\rfloor$.
Then, uniformly for $\theta\in[0,1]$,
\[
-(t-k)\log M = -(1-\theta)mt^2+o(t^2),\qquad
\binomtwo{k}\log(1-p)=\frac{\theta^2}{2}\log(1-p)\,t^2+o(t^2),
\]
and
\[
\frac{6\lambda(p)}{D}\cdot\frac{\binomthree{k}}{t}
=\frac{\lambda(p)}{D}\theta^3 t^2+o(t^2).
\]
dividing by $t^2$. Also, by \eqref{eq:m-def} we may replace $m$ by $-\frac12\log p+\frac{\kappa(p)}{D}-\frac{K}{D^2}-\delta$ which is exactly \eqref{eq:alpha-error}.
\end{proof}

\section{Proof of Theorem~\ref{cor:main-numerics}}
In this section we prove Theorem~\ref{cor:main-numerics} by approximating $\alpha(p,D)$ for $D$ large.
\begin{proof}[Proof of Theorem~\ref{cor:main-numerics}]
We consider
\[
f(\theta)
=
(1-\theta)\Big(-\tfrac12\log p+\tfrac{\kappa(p)}{D}\Big)
-\frac{\theta^2}{2}\log(1-p)
-\frac{\lambda(p)}{D}\theta^3
-(2-\theta)\frac{K}{D^2},
\qquad \theta\in[0,1], 
\]
and notice that $$\alpha(p,D)=\min_{\theta\in [0,1]} f(\theta)\, .$$
Then
\[
f'(\theta)
=
-\Big(-\tfrac12\log p+\tfrac{\kappa(p)}{D}\Big)
-\theta\log(1-p)
-\frac{3\lambda(p)}{D}\theta^2
+\frac{K}{D^2}.
\]
Setting $f'(\theta)=0$ and multiplying by $-1$ yields the quadratic equation
\[
\frac{3\lambda(p)}{D}\theta^2+\log(1-p)\,\theta
+\Big(-\tfrac12\log p+\tfrac{\kappa(p)}{D}-\tfrac{K}{D^2}\Big)=0.
\]
Hence the stationary points are
\[
\theta
=
\frac{D}{6\lambda(p)}
\left(
-\log(1-p)\pm
\sqrt{\log(1-p)^2-\frac{12\lambda(p)}{D}\left(-\tfrac12\log p+\tfrac{\kappa(p)}{D}-\tfrac{K}{D^2}\right)}
\right).
\]
For large $D$, the ``$+$'' root is larger, therefore the minimum of $f$ is the $``-''$ root. Writing this root as $\theta_*$, we have
\[
\theta_*
=
\frac{D}{6\lambda(p)}
\left(
-\log(1-p)-
\sqrt{\log(1-p)^2-\frac{12\lambda(p)}{D}\left(-\tfrac12\log p+\tfrac{\kappa(p)}{D}-\tfrac{K}{D^2}\right)}
\right).
\]

\medskip
\noindent\textbf{Large-$D$ expansion.}
Let $\theta_*=\theta_0+\theta_1/D+O(1/D^2)$. At $D=\infty$ the quadratic reduces to
$\log(1-p)\,\theta_0-\tfrac12\log p=0$, hence
\[
\theta_0=\frac{\log p}{2\log(1-p)}.
\]
Expanding the quadratic to first order in $1/D$ gives
\[
\log(1-p)\,\theta_1+3\lambda(p)\theta_0^2+\kappa(p)=0,
\]
so
\[
\theta_1=-\frac{\kappa(p)+3\lambda(p)\theta_0^2}{\log(1-p)}
=
\frac{1}{\log(1-p)}\left(
-\kappa(p)-\frac{3(\log p)^2}{4\log(1-p)^2}\lambda(p)
\right).
\]
Therefore,
\[
\boxed{
\theta_*
=
\frac{\log p}{2\log(1-p)}
+
\frac{1}{D\,\log(1-p)}
\left(
-\kappa(p)-\frac{3(\log p)^2}{4\log(1-p)^2}\lambda(p)
\right)
+O_p\!\left(\frac{1}{D^2}\right).
}
\]
Note that the parameter $K$ only enters the quadratic as $-K/D^2$, and therefore does not contribute at
order $1/D$ (it first appears in the $1/D^2$ term).

Now substituting in $f(\theta_*)$ we get $$\alpha(p,D)=\frac{(\log p-4\log(1-p))\log p}{8 \log(1-p)}+\frac{1}{D}\left(\left(1-\frac{\log p}{2 \log(1-p)}\right)\kappa(p)-\frac{(\log p)^3}{8 (\log (1-p))^3}\lambda(p)\right)+O_p\!\left(\frac{1}{D^2}\right)\, .$$
Finally substituting in the values of $\kappa(p)$ and $\lambda(p)$ we obtain 
$$\alpha(p,D)-\frac{(\log p-4\log(1-p))\log p}{8 \log(1-p)}=\frac{a^3}{D}\left(\left(1-\frac{\log p}{2 \log(1-p)}\right)\frac{1}{p^3}-\frac{(\log p)^3}{8 (\log (1-p))^3}\frac{1}{(1-p)^3}\right)+O_p\!\left(\frac{1}{D^2}\right)$$

Finally $$D\gamma=\left(1-\frac{\log p}{2 \log(1-p)}\right)\frac{1}{p^3}-\frac{(\log p)^3}{8 (\log (1-p))^3}\frac{1}{(1-p)^3}+O_p\left(\frac{1}{D}\right)>0$$ for all $p\in (0.42,0.5)$ and $D$ large enough as we can see ploted in the graph below. 

\begin{center}
  \includegraphics[width=0.8\textwidth]{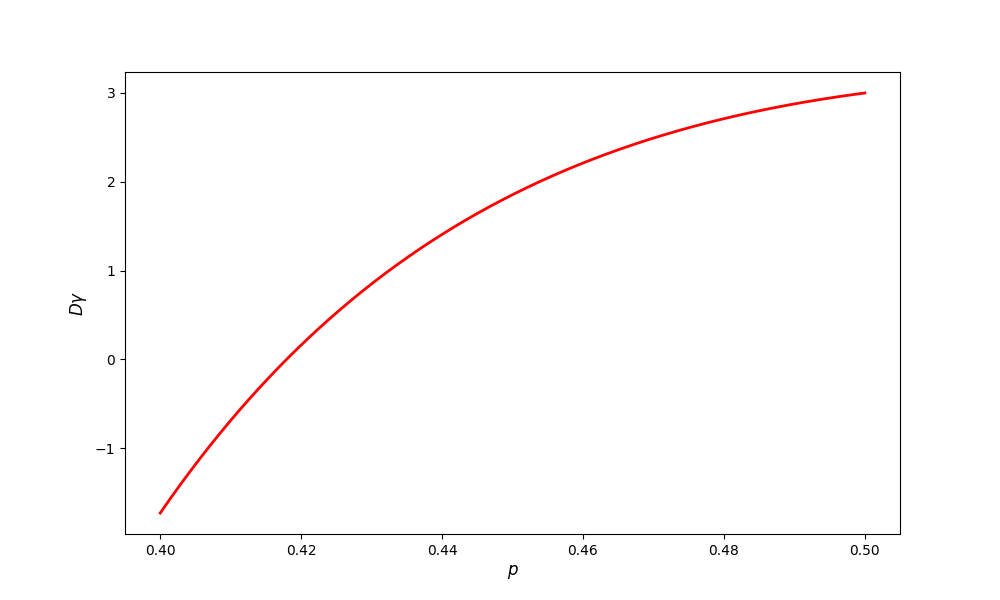}
\end{center}
\end{proof}

\end{document}